\documentclass{amsart}
\usepackage{amsfonts}

\newtheorem{theorem}{Theorem}[section]

\newtheorem{proposition}[theorem]{Proposition}
\theoremstyle{definition}

\theoremstyle{remark}
\newtheorem{remark}[theorem]{Remark}
\numberwithin{equation}{section}

\begin{document}
\title{The heat kernel of a weighted Maass-Laplacian with real weights}
\author{Zenan \v{S}abanac}
\address{Department of Mathematics, University of Sarajevo, Zmaja od Bosne
35, 71 000 Sarajevo}
\email{zsabanac@pmf.unsa.ba}
\author{Lamija \v{S}\'{c}eta}
\address{School of Economics and Business, University of Sarajevo, Alije
Izetbegovi\'{c}a 1, 71 000 Sarajevo}
\email{lamija.sceta@efsa.unsa.ba}

\begin{abstract}
We derive closed formula for the heat kernel $K_\mathbb{H}$ associated to Maass-Laplacian operator $D_k$ for any real $k$ and prove that heat kernel $K_\mathbb{H}$ is strictly monotone decreasing function. We also derive some important asymptotic formulae for the heat kernel $K_\mathbb{H}$.
\end{abstract}

\maketitle







\section{Introduction}

Let $\Gamma\subset PSL\left(2,\mathbb{R}\right)$ be a discontinuous subgroup
of the group of all real M\"{o}bius transformations acting on the upper
half-plane $\mathbb{H}$, and by $\hat{\Gamma}$ denote the group covering $%
\Gamma$ under the projection $SL\left(2,\mathbb{R}\right)\rightarrow
PSL\left(2,\mathbb{R}\right)$. We assume that $-I \in \hat{\Gamma}$, where $I$ denotes the identity element.

We assume that the fundamental domain of $\Gamma$ has finite hyperbolic area, in which case the Riemann surface $X_\Gamma$ identified with $\mathbb{H}\backslash \Gamma$ has a finite volume.

The spectral theory of the Laplace operator and Laplace-like operators (Maass Laplacians) on Riemannian surfaces is an important object of study with vast applications in theory of automorphic forms, string theory, scattering theory and differential geometry, to name a few.

The study of the determinant of the Laplacian acting on the space of twice continuously differentiable complex functions on a Riemann surface (actually, the study of its self-adjoint extension) is crucial in the Polyakov's string theory \cite{polyakov}. Spectral theory of the hyperbolic Laplacian
$$D_0:=-y^2\left(\frac{\partial ^2}{\partial x^2}+\frac{\partial^2}{\partial y^2}\right)$$
and the related Selberg zeta function is a crucial ingredient in the construction of the determinant of the Lax-Phillips scattering operator, introduced in \cite{LaxPhi} (see \cite{FJS20} and the references therein).

The analytic and Rademeister torsion on Riemannian manifolds is related to the spectral theory of the associated Laplacian, see \cite{RaySinger}, so the spectral theory of the Laplacian can be viewed as a tool to deduce geometric invariants.

When the manifold is compact, the study of the determinant of the Laplacian and the study of the analytic torsion and related invariants from the abovementioned references relies on the study of the spectral zeta function, defined as the Mellin transform of the corresponding heat kernel.

More recently, the fact that for any integral or half-integral $k$, the kernel of the operator $D_k - k(1-k)$ is isomorphic to the space of weight $2k$ cusp forms for $\Gamma$ (when both spaces are viewed as $\mathbb{C}-$vector spaces) was used in \cite{FJK16} and \cite{FJK19} to deduce effective sup-norm bounds on average for such cusp forms. Here $D_k$ denotes the weighted Maass-Laplacian, defined by
\begin{equation} \label{Dk defin}
D_k:=-y^2\left(\frac{\partial ^2}{\partial x^2}+\frac{\partial^2}{\partial y^2%
}\right)+2iky\frac{\partial}{\partial x}.
\end{equation}
The asymptotic behavior of the heat kernel of the Maass-Laplacian $D_k$ on the upper half-plane, evaluated for $k\in\frac{1}{2}\mathbb{Z}$ in a closed form by \cite{Oshima} played a crucial role in the results of \cite{FJK16} and \cite{FJK19}.

The operator $D_k$ acts on twice continuously differentiable functions $f:\mathbb{H}\to\mathbb{C}$ satisfying transformation behavior \eqref{transform of f} with respect to a certain multiplier system of weight $k$. Multiplier systems and weighted Maass-Laplacians can be defined for arbitrary real weights, not just for integral and half-integral weights, see the construction on pp. 335--337 of \cite{He83}. Moreover, there are examples of important multiplier systems with weight $k\notin\frac{1}{2}\mathbb{Z}$, constructed from generalized Dedekind sums associated to Fuchsian groups (see e.g. \cite{Bu, JO'SS, BJO'SS20}). Therefore, it is of interest to deduce a closed formula for the heat kernel associated to the weighted Maass-Laplacian $D_k$ on $\mathbb{H}$ for any real $k$ and establish its asymptotic behavior. Despite the importance of the heat kernel, this was not done in prior research papers, to the best of our knowledge.

In this paper we derive a closed formula for the heat kernel associated to $D_k$ for an arbitrary weight $k\neq 0$ (the case when $k=0$ is well-known) and prove that it is strictly monotone decreasing when viewed as a function of the hyperbolic distance between the points in $\mathbb{H}$. We further investigate asymptotic behavior of the heat kernel as $t\downarrow 0$ and $t \to \infty$. We prove that the heat kernel decays exponentially as $t\downarrow 0$ and that the rescaled kernel is integrable in $t$ on $[1,\infty)$. This result is important in various applications, e.g. one may use results of \cite{JLa03} and obtain a closed formula for a branched meromorphic continuation of the Poisson kernel associated to $D_k$. Also, one may follow the pattern described in \cite{Mull83} to deduce meromorphic continuation of the spectral zeta function. Since the heat kernel is a fundamental object in various mathematical disciplines, we are certain that a closed formula derived in this paper for all real weights $k$, together with description of its behavior both as a function of the hyperbolic distance and as a function of time $t$, will have many applications.

\section{Preliminaries}

Let $\mathbb{H}=\{z\vert \ z=x+iy, y>0\}$ be the upper half-plane with
Poincare metric $dz=\frac{\vert dz \vert}{y}$ and areal measure $d\mu(z)=\frac{dxdy}{y^2}$ expressed in the rectangular coordinates. By $r(z,w)$ we denote the hyperbolic distance between points $z,w\in\mathbb{H}$.

For $S=\left(%
\begin{array}{cc}
a_S & b_S \\
c_S & d_S%
\end{array}%
\right)\in SL(2,\mathbb{R})$ and $z\in\mathbb{H}$ let $j(S,z):=c_Sz+d_S$. Then, for any two matrices $S,T\in SL(2,\mathbb{R})$, there exist a unique number $w(S,T)\in\{-1,0,1\}$, the \emph{phase factor} such that
$$
2\pi i w(S,T)=- \log j(ST,z)+\log j(S,Tz) + \log j(T,z),
$$
where logarithmic function on the right hand side is the principal branch normalized so that the argument takes values in $(-\pi,\pi]$, see e.g. \cite[Section 2.2]{JO'SS}.

A (scalar-valued) \emph{multiplier system} $\nu_k$ of a real weight $k$ on $\hat{\Gamma}$ is a function $\nu_k:\hat{\Gamma} \to S^1$, where $S^1$ is a unit circle in $\mathbb{C}$ with the properties:
\begin{enumerate}
\item[i)] $\nu_k(-I)=\exp(-2\pi i k)$,
\item[ii)] $\nu_k(ST)=\sigma_{2k}(S,T)\nu_k(S)\nu_k(T)$, for $S,T\in\hat{\Gamma}$,
\end{enumerate}
where $\sigma_{2k}(S;T):=\exp(4\pi i k w(S,T))$ is the \emph{factor system} of weight $2k$, see \cite[Definition 1.3.1]{Fi87}.

We denote by $\mathfrak{F}_k$ the space of all functions $f:\mathbb{H}\rightarrow \mathbb{C}$, which transform as
\begin{equation} \label{transform of f}
f\left(S z\right)=\nu_k\left(S \right)\left(\frac{c z+d}{c \overline{z}%
+d}\right)^kf\left(z\right),
\end{equation}
for all $S =\left(%
\begin{array}{cc}
a & b \\
c & d%
\end{array}%
\right)\in\hat{\Gamma}$, where $\nu_k$ is the multiplier system of weight $k$ on $\hat{\Gamma}$.

We denote by $\mathfrak{H}_k$ the Hilbert space of all functions $f\in\mathfrak{F}_k$ that are square-integrable on the fundamental domain of $\Gamma$, with respect to the hyperbolic measure $d\mu(z)$.

The weighted Maass-Laplacian $D_k$, defined by \eqref{Dk defin} preserves the transformation behavior \eqref{transform of f} for all twice continuously differentiable functions from $\mathfrak{F}_k$ and acts on a dense subspace of $\mathfrak{H}_k$. It is proved in \cite{roelcke1} that $D_k$ extends to a self-adjoint operator on $\mathfrak{H}_k$.

Let us note here that weighted Maass-Laplacian $D_k$ is a specialization of the operator
$$
\Delta_{\alpha,\beta}= -y^2\left( \frac{\partial^2}{\partial x^2} + \frac{\partial^2}{\partial y^2} \right) + (\alpha - \beta)iy\frac{\partial}{\partial x} - (\alpha +\beta)y\frac{\partial}{\partial y}
$$
at $\alpha+\beta=0$, first studied by Maass in \cite{M52}.

\emph{The heat kernel} of the Maass-Laplacian $D_k$ on the upper half-plane, denoted by $K_{\mathbb{H},k}\left( t; z,w\right)$, where $t>0$ is real and $z,w\in\mathbb{H}$ is a unique fundamental solution to the differential operator $D_k+\partial_t$, satisfying a Dirac condition as $t\downarrow 0$:
$$
f(z)=\lim_{t\downarrow 0}\int_{\mathbb{H},k}K_{\mathbb{H}}(t;z,w)d\mu(w),
$$
for any bounded and continuous function $f$ on $\mathbb{H}$, where the convergence is uniform on compact subsets of $\mathbb{H}$ (see \cite{JLa03}).

\section{Closed formula for the heat kernel associated to Maass-Laplacian operator $D_{k}$}

In his seminal paper \cite{Fay}, Fay initiated the study of the heat kernel of $D_k$ on $\mathbb{H}$, however,
D'Hocker and Phong \cite{Phong} established that this formula is not completely correct. In case when $k\in\frac{1}{2}\mathbb{Z}$, the correct expression for the heat kernel $K_{\mathbb{H},k}\left( t; z,w\right)$ of the Maass-Laplacian $\Delta_k$ on the upper half-plane was deduced by Oshima in \cite{Oshima}.

It is well known fact that the heat kernel is a radial function, i.e. it depends only upon the hyperbolic distance $r=r(z,w)$ between points $z$ and $w$ in $\mathbb{H}$. Moreover, the weight $k\in \mathbb{R}$ is arbitrary but fixed throughout the paper. Therefore, to ease the notation, we will denote the heat kernel $K_{\mathbb{H},k}\left( t; z,w\right)$ simply by $K_{\mathbb{H}}\left( t; r\right)$.

In this section we extend Oshima's result to all real values of $k$ and prove the following theorem.

\begin{theorem}
The heat kernel for the Maass-Laplacian $D_{k}$, for every $k\in \mathbb{R}$,
is given by
\begin{equation}\label{heat k main formula}
K_{\mathbb{H}}\left( t,r\right) =\frac{e^{-\frac{t}{4}}}{\left( 4\pi
t\right) ^{\frac{3}{2}}}\int\limits_{r}^{\infty }\frac{ue^{-\frac{u^{2}}{4t}}}{%
\sqrt{2\cosh u-2\cosh r}}\cdot \mathcal{T}_{2k}\left( \frac{\cosh \frac{u}{2}}{\cosh
\frac{r}{2}}\right) du,
\end{equation}
where
\begin{equation*}
\mathcal{T}_{2k}\left( x\right) =\left( x+\sqrt{x^{2}-1}\right) ^{2k}+\left( x-\sqrt{%
x^{2}-1}\right) ^{2k}.
\end{equation*}
\end{theorem}

\begin{proof}
Our starting point is \cite[formula (2.14)]{Oshima}
\begin{equation}\label{Osh f-la}
g(r)= -\frac{1}{\pi i}\int\limits_{\mathrm{Re}s=\alpha >|k|}h(s)\left(s-\frac{1}{2}\right)Q_{s,k}(\cosh r) ds,
\end{equation}
where $h(s)$ is a test function satisfying certain conditions and
\begin{equation*}
Q_{s,k}\left(\cosh r\right)= -\frac{1}{4\pi}\int\limits_r^{+\infty}e^{-\left(s-\frac{1}{2}\right)u}\frac{I_k(u,r)}{\sqrt{2\cosh u-2\cosh r}}du
\end{equation*}
is given by \cite[formula (21)]{Fay} with $n=0$. Here
\begin{equation}\label{bracket}
\begin{split}
I_{k}\left( u,r\right) = \left( \cosh r-1\right) ^{-k} \cdot \left[
e^{2k\Theta }\left( \sqrt{2}\sinh \frac{u}{2}-\sqrt{\cosh u-\cosh r}\right)
^{2k}\right.   \\
 + \left. e^{-2k\Theta }\left( \sqrt{2}\sinh \frac{u}{2}+\sqrt{\cosh u-\cosh r}%
\right) ^{2k}\right],
\end{split}
\end{equation}
where $\Theta$ satisfies the relation
\begin{equation} \label{theta rel}
e^{\pm \Theta }\sinh r=e^{u}-\cosh r\pm e^{\frac{u}{2}}\sqrt{2\cosh u-2\cosh r},
\end{equation}
see \cite[p. 157]{Fay}.

\noindent Formula \eqref{Osh f-la} is derived from the spectral decomposition theorem (\cite[Theorem 1.5]{Fay}) and holds true for any $k\in\mathbb{R}$. It is known that the heat kernel $K_{\mathbb{H}}\left(t,r\right)$ is deduced by taking $h\left(s\right)=e^{s\left(s-1\right)t}$ in \eqref{Osh f-la} and integrating along the vertical line $ \mathrm{Re}s=\alpha$, for any $\alpha >|k|$. Therefore, application of Fubini-Tonelli theorem in \eqref{Osh f-la} yields
\begin{equation} \label{K pre-formula}
K_{\mathbb{H}}\left(t,r\right)=\frac{1}{4\pi^2 i}\int\limits_r^{+\infty}H(u,t)\frac{I_k(u,r)}{\sqrt{2\cosh u-2\cosh r}}du
\end{equation}
where
\begin{equation*}
H(u,t)=\int\limits_{\mathrm{Re}s=\alpha >|k|}e^{s(s-1)t-\left(s-\frac{1}{2}\right)u}\left(s-\frac{1}{2}\right)ds.
\end{equation*}
Moving the line of integration in the above integral to the line $\mathrm{Re}s=1/2$, which is justified due to holomorphicity of the function under the integral sign and setting $s=1/2+iy$ we get
\begin{equation*}
H(u,t)=-e^{-\frac{t}{4}}\int\limits_{-\infty}^{+\infty}e^{-ty^2-iuy}ydy.
\end{equation*}
Applying \cite[formula 3.462.6.]{GR07} with $p=t>0$ and $q=-iu/2$ we deduce that
\begin{equation*}
H(u,t)=\frac{iu}{2t}\sqrt{\frac{\pi}{t}}e^{-\frac{t}{4}}e^{-\frac{u^2}{4t}}.
\end{equation*}
Inserting this into \eqref{K pre-formula} we get
\begin{equation}  \label{K}
K_{\mathbb{H}}\left(t,r\right)=\frac{e^{-\frac{t}{4}}}{\left(4\pi t\right)^{%
\frac{3}{2}}}\int_{r}^{\infty}\frac{ue^{-%
\frac{u^2}{4t}}I_k(u,r)}{\sqrt{2\cosh u -2\cosh r}} du.
\end{equation}
To complete the proof of the theorem it remains to prove that $I_k\left(u,r\right)=\mathcal{T}_{2k}\left( \frac{\cosh \frac{u}{2}}{\cosh\frac{r}{2}}\right)$.

\noindent Using \eqref{bracket}, \eqref{theta rel} and formula $\sinh ^{2}r=\left( \cosh r-1\right) \left( \cosh r+1\right)$ we get
\begin{equation}  \label{bracket2}
\begin{gathered}
I_k\left(u,r\right)=\frac{1}{\left(\cosh r-1\right)^{2k}\left(\cosh
r+1\right)^{k} }\cdot \\
\left[\left(e^u-\cosh r+e^{\frac{u}{2}}\sqrt{2\cosh u-2\cosh r}%
\right)^{2k}\left(\sqrt{2}\sinh \frac{u}{2}-\sqrt{\cosh u -\cosh r}%
\right)^{2k} \right. \\
\left.+\left(e^u-\cosh r-e^{\frac{u}{2}}\sqrt{2\cosh u-2\cosh r}%
\right)^{2k}\left(\sqrt{2}\sinh \frac{u}{2}+\sqrt{\cosh u -\cosh r}%
\right)^{2k}\right].
\end{gathered}
\end{equation}
Elementary transformations based on the properties of functions $\sinh$ and $\cosh$ yield that the first summand in the square bracket \eqref{bracket2} is
\begin{equation}\label{summand1}
\begin{gathered}
\left[\left(e^u-\cosh r+e^{\frac{u}{2}}\sqrt{2\cosh u-2\cosh r}
\right)\cdot\left(\sqrt{2}\sinh \frac{u}{2}-\sqrt{\cosh u -\cosh r}\right)
\right]^{2k}
\\=\left(\cosh r-1\right)^{2k}\left(\sqrt{2}\cosh\frac{u}{2}+\sqrt{\cosh
u-\cosh r}\right)^{2k}.
\end{gathered}
\end{equation}
Analogously, the second summand in the square bracket of \eqref{bracket2} is
\begin{equation}\label{summand2}
\begin{gathered}
\left[\left(e^u-\cosh r-e^{\frac{u}{2}}\sqrt{2\cosh u-2\cosh r}%
\right)\cdot\left(\sqrt{2}\sinh \frac{u}{2}+\sqrt{\cosh u -\cosh r}\right)%
\right]^{2k}
\\=\left(\cosh r-1\right)^{2k}\left(\sqrt{2}\cosh\frac{u}{2}-\sqrt{\cosh
u-\cosh r}\right)^{2k}.
\end{gathered}
\end{equation}
Inserting \eqref{summand1} and \eqref{summand2} into %
\eqref{bracket2}, we get
\begin{equation*}
\begin{split}
I_{k}\left( u,r\right) =\frac{1}{\left( \cosh r+1\right) ^{k}} \cdot \left[ \left( \sqrt{2}\cosh \frac{u}{2}+\sqrt{%
\cosh u-\cosh r}\right) ^{2k}\right. \\
\left. +\left( \sqrt{2}\cosh \frac{u}{2}-\sqrt{\cosh u-\cosh r}\right) ^{2k}%
\right].
\end{split}
\end{equation*}
Using that $\cosh r+1=2\cosh ^{2}(r/2)$ we easily deduce the equality
\begin{equation*}
I_{k}\left( u,r\right) = \mathcal{T}_{2k}\left( \frac{\cosh \frac{u}{2}}{\cosh\frac{r}{2}}\right).
\end{equation*}
The proof is complete.
\end{proof}

\begin{remark}
When $k\in\frac{1}{2}\mathbb{Z}$, the function $\mathcal{T}_{2k}(x)$ equals twice the $2k$th Chebyshev polynomial, hence our result indeed generalizes the main result of \cite{Oshima}.
\end{remark}

\section{Properties of the heat kernel}

In this section we investigate properties of the heat kernel as a function of the hyperbolic distance $r$ and as a function of the time $t>0$.

\subsection{Behavior of $K_{\mathbb{H}}\left(t,r\right)$ as a function of $r$}

Monotonicity of the heat kernel for the Maass-Laplacian, as a function of the hyperbolic distance $r$ is proven by Friedman at al. in \cite[Proposition 3.2]{FJK16}, assuming that $k\in \frac{1}{2}\mathbb{Z}$.  The following proposition generalizes that result to the case of all real weights $k$.

\begin{proposition}
For any $t>0$, the heat kernel $K_{\mathbb{H}}\left(t,r\right)$ on $\mathbb{H%
}$ associated to $D_k$, $k\in\mathbb{R}$ is strictly monotone decreasing as a function of
$r>0$.
\end{proposition}

\begin{proof}
The proof follows the lines of the proof of \cite[Proposition 3.2]{FJK16}, so we will focus on parts that use the function $\mathcal{T}_{2k}$ instead of the Chebyshev polynomial $T_{2k}$. We start with by writing the heat kernel $K_{\mathbb{H}}\left(t,r\right) $ as
\begin{equation*}
K_{\mathbb{H}}\left( t,r\right) =\frac{e^{-\frac{t}{4}}}{\sqrt{2}\left( 4\pi
t\right) ^{\frac{3}{2}}}\int_{r}^{\infty }\frac{\sinh u}{\sqrt{\cosh u-\cosh
r}}F_{k}\left( t,r,u\right) du,
\end{equation*}%
where
\begin{equation*}
F_{k}\left( t,r,u\right) =\frac{ue^{-\frac{u^{2}}{4t}}}{\sinh u}\cdot
\mathcal{T}_{2k}\left( \frac{\cosh \frac{u}{2}}{\cosh \frac{r}{2}}\right) .
\end{equation*}
Integration by parts yields:
\begin{equation*}
K_{\mathbb{H}}\left( t,r\right) =-\frac{\sqrt{2}e^{-\frac{t}{4}}}{\left(
4\pi t\right) ^{\frac{3}{2}}}\int_{r}^{\infty }\frac{\partial }{\partial u}%
F_{k}\left( t,r,u\right) \sqrt{{\cosh u-\cosh r}}du.
\end{equation*}%
In order to prove monotonicity of $K_{\mathbb{H}}\left( t,r\right) $ we need to prove that $\frac{\partial}{\partial r}K_{\mathbb{H}}\left( t,r\right)<0 $ for all $r>0$. Proceeding analogously as in \cite{FJK16}, we get:
\begin{equation}\label{delta K}
\begin{split}
\frac{\partial }{\partial r}K_{\mathbb{H}}\left( t,r\right) =\frac{e^{-\frac{%
t}{4}}}{\sqrt{2}\left( 4\pi t\right) ^{\frac{3}{2}}}\int_{r}^{\infty }\left(
\sinh u\frac{\partial }{\partial r}F_{k}\left( t,r,u\right) +\sinh r\frac{%
\partial }{\partial u}F_{k}\left( t,r,u\right) \right)
\\ \cdot \frac{du}{\sqrt{{%
\cosh u-\cosh r}}}.
\end{split}
\end{equation}
Therefore, to prove the proposition it suffices to prove the inequality
\begin{equation}\label{main ineq}
\sinh u\frac{\partial }{\partial r}F_{k}\left( t,r,u\right) +\sinh r\frac{%
\partial }{\partial u}F_{k}\left( t,r,u\right) <0,\quad \text{for all  } r>0.
\end{equation}
It is trivial to deduce that
\begin{equation}\label{bracket 4}
\begin{split}
\sinh u\frac{\partial }{\partial r}F_{k}\left( t,r,u\right) +\sinh r\frac{%
\partial }{\partial u}F_{k}\left( t,r,u\right) =\sinh r\left( \frac{1}{u}-%
\frac{u}{2t}-\frac{\cosh u}{\sinh u}\right) \cdot F_{k}\left( t,r,u\right)
 \\
+\mathcal{T}_{2k}^{\prime }\left( x\right) \cdot \frac{ue^{-\frac{u^{2}}{4t}}\sinh
\frac{r}{2}}{2\cosh \frac{u}{2}\cosh ^{2}\frac{r}{2}}\left( \cosh ^{2}\frac{r%
}{2}-\cosh ^{2}\frac{u}{2}\right),
\end{split}
\end{equation}
where
\begin{equation} \label{T prime}
\mathcal{T}_{2k}^{\prime }\left( x\right) =\frac{2k}{\sqrt{x^{2}-1}}\left( \left( x+%
\sqrt{x^{2}-1}\right) ^{2k}-\left( x-\sqrt{x^{2}-1}\right) ^{2k}\right)
\end{equation}
and $x=\frac{\cosh \frac{u}{2}}{\cosh \frac{r}{2}}$.

\noindent For $u>0$ it holds $\tanh u\leq u$, hence $\frac{\cosh u}{\sinh u}=\coth
u\geq \frac{1}{u}$. Therefore,
\begin{equation*}
\frac{1}{u}-\frac{u}{2t}-\frac{\cosh u}{%
\sinh u}\leq -\frac{u}{2t}<0.
\end{equation*}
It is easy to check that $\mathcal{T}_{2k}^{\prime
}(x)>0$ for all $x=\frac{\cosh \frac{u}{2}}{\cosh \frac{r}{2}}>1$. Namely, when $k>0$ both factors on the right-hand side of \eqref{T prime} are positive, while, for $k<0$ both factors are negative.

\noindent Trivially, $\cosh ^{2}\frac{r}{2}-\cosh ^{2}\frac{u}{2}<0
$ for $u>r>0$, since $\cosh ^{2}x$ is increasing function for $x>0$. This proves that each term in the sum on the right-hand side of \eqref{bracket 4} is negative, hence \eqref{main ineq} holds true.

\noindent The proof is complete.
\end{proof}

\subsection{Behavior of $K_{\mathbb{H}}\left(t,r\right)$ as a function of $t$}

The small-time and large-time asymptotic behavior of the heat kernel is essential for its applications in physics, since the meromorphic continuation of the spectral zeta function associated to eigenvalues of the Maass-Laplacian relies on this information, see e.g. \cite{Mull83}.

The following proposition describes the asymptotic behavior of $K_{\mathbb{H}}\left( t,r\right)$ as $t\downarrow 0$ and $t\to \infty$.
\begin{proposition}
The heat kernel $K_{\mathbb{H}}\left(t,r\right)$ on $\mathbb{H%
}$ associated to $D_k$, $k\in\mathbb{R}$, for all $r>0$ satisfies the following two relations:
\begin{equation}\label{small t}
K_{\mathbb{H}}\left( t,r\right)= O\left(t^{-\frac{3}{2}}e^{-\frac{r^2}{4t}} \right), \quad \text{as  } t\downarrow 0,
\end{equation}
where the implied constant is independent of $t$, and
\begin{equation}\label{large t}
K_{\mathbb{H}}\left( t,r\right)= O\left(e^{-\frac{t}{4}}t^{-\frac{3}{2}} \right), \quad \text{as  } t \to \infty,
\end{equation}
where the implied constant is independent of $t$.
\end{proposition}

\begin{proof}
We start by writing the heat kernel as
$$
K_{\mathbb{H}}\left( t,r\right) = \frac{e^{-\frac{t}{4}} e^{-\frac{r^2}{4t}}}{\left( 4\pi
t\right) ^{\frac{3}{2}}}\int\limits_{r}^{\infty }\frac{u e^{-\frac{u^2-r^2}{4t}}}{
\sqrt{2\cosh u-2\cosh r}}\cdot \mathcal{T}_{2k}\left( \frac{\cosh \frac{u}{2}}{\cosh
\frac{r}{2}}\right) du.
$$
The integral on the right-hand side of the above inequality can be written as the sum of two integrals $I_1(t,r)$, where the integration is along the interval $(r,r+1)$, and $I_2(t,r)$, where the integration is taken along the interval $(r+1,\infty)$.
First, we will show that the integral $I_1(t,r)$ is finite and bounded by a constant independent of $t$. Namely, after substitution $x= \frac{\cosh \frac{u}{2}}{\cosh \frac{r}{2}}$, applying the inequality $e^{-\frac{u^2-r^2}{4t}}\leq 1$, we get
\begin{equation*}
I_1(t,r)\leq\int\limits_{1}^{a}\frac{\mathrm{arcosh}(x\cosh\frac{r}{2})}{\sqrt{x^{2}-1}}\mathcal{T}_{2k}(x) \frac{dx}{\sqrt{x^{2} \cosh^{2} \frac{r}{2}-1}}
\end{equation*}
where $a=\cosh \frac{(r+1)}{2}\cosh^{-1} \frac{r}{2}$. Since $x\geq 1$ it obviously holds
\begin{equation*}
I_1(t,r)\leq \frac{1}{\sqrt{2(\cosh^2\frac{r}{2}-1)}}\int\limits_{1}^{a}\frac{\mathrm{arcosh}(x\cosh \frac{r}{2})}{\sqrt{x-1}}\mathcal{T}_{2k}(x) dx.
\end{equation*}
Taking into account that $\frac{\mathrm{arcosh}(x\cosh(r/2))}{\sqrt{x-1}}\mathcal{T}_{2k}(x) = O\left(\left(x-1\right)^{-1/2}\right)$ as $x\searrow 1$ and $\left(x-1\right)^{-1/2}$ is integrable on $(1,a)$, we see that the integral on the right hand-side is finite. Hence, we have proved that $I_1(t,r)$ is bounded.

\noindent Next, we estimate $I_2(t,r)$. When $u>r+1$, there exists an absolute constant $C$, depending upon $r$, such that $(\cosh u - \cosh r)^{-1/2} \leq C e^{-u/2}$. Since $\mathcal{T}_{2k}(x)$ increases for $x>1$, from $\cosh (r/2)\geq 1$, we immediately deduce that
$$
\mathcal{T}_{2k}\left(\frac{\cosh \frac{u}{2}}{\cosh \frac{r}{2}}\right)\leq \mathcal{T}_{2k}\left(\cosh \frac{u}{2}\right)\leq 2^{2|k|+2}e^{|k|u}.
$$
Therefore,
\begin{equation}\label{I2}
I_2(t,r)\ll \int\limits_{r+1}^{\infty }ue^{-\frac{u^2-r^2}{4t}}e^{\left(|k|-\frac{1}{2}\right)u}du,
\end{equation}
where the implied constant is independent of $t$. When $|k|-1/2<0$, the above integral is obviously bounded by a constant which is independent of $t$, so it is left to estimate it in the case $|k|-1/2 \geq 0$.

\noindent Until this point, we did not use any restrictions on $t$. Now, we will distinguish the cases $t\downarrow 0$ and $t\to\infty$.

First, assume that $t\downarrow 0$, then, we may assume $0<t<\left(4(|k|-1/2)+4\right)^{-1}$, hence
$$
I_2(t,r)\ll \int\limits_{r+1}^{\infty }ue^{-(u^2-r^2)}e^{\left(|k|-\frac{1}{2}\right)(r^2+\frac{1}{4}-(u-\frac{1}{2})^2)}du,
$$
where the implied constant is independent of $t$. The integral on the right-hand side is obviously finite. Therefore, for $0<t<\left(4(|k|-1/2)+4\right)^{-1}$ it yields that $I_2(t,r)$, and, hence, $I_1(t,r)+I_2(t,r)$ is bounded uniformly in $t$, which proves \eqref{small t}.

Next, we assume that $t$ is large. Since the function under the integral sign on the right hand-side of \eqref{I2} is non-negative, for $t>1$ we have
$$
I_2(t,r)\ll e^{\frac{r^2}{4t}} \int\limits_{0}^{\infty }ue^{-\frac{u^2}{4t}+\left(|k|-\frac{1}{2}\right)u}du.
$$
The above integral can be explicitly computed as a function of $t$, $|k|$ and $r$ using \cite[formula 3.462.1.]{GR07} with $\beta=1/(4t)>0$, $\gamma=-(|k|-1/2)$ and $\nu=2>0$ to get
\begin{equation}\label{I2 new b}
I_2(t,r)\ll e^{\frac{r^2}{4t}} 2te^{\frac{t}{2}\left(|k|-\frac{1}{2}\right)^2}D_{-2}\left(-\sqrt{2t}\left(|k|-\frac{1}{2}\right)\right),
\end{equation}
where $D_{-\nu}(z)$ denotes the parabolic cylinder function (see, e.g. \cite[p. 1037]{GR07} for a definition). Using the asymptotic expansion of the parabolic cylinder function $D_{-2}(-\sqrt{2t}(|k|-1/2))$, with $|z|=\sqrt{2t}(|k|-1/2)$ large (\cite[formula 9.246.1.]{GR07}), we deduce that
$$
D_{-2}\left(-\sqrt{2t}\left(|k|-\frac{1}{2}\right)\right)=\frac{1}{2t\left(|k|-\frac{1}{2}\right)^2}e^{-\frac{t}{2}\left(|k|-\frac{1}{2}\right)^2}\left(1+O(t^{-1})\right).
$$
Inserting this into \eqref{I2 new b}, we immediately deduce that for large enough $t$ the integral $I_2(t,r)$ is bounded by a product of $e^{\frac{r^2}{4t}}$ and a function which is uniformly bounded in $t$. Since
$$
K_{\mathbb{H}}\left( t,r\right) =\frac{e^{-\frac{t}{4}}}{\left( 4\pi
t\right) ^{\frac{3}{2}}}\left(I_1(t,r)+I_2(t,r)\right),
$$
this proves \eqref{large t}.

\noindent The proof is complete.
\end{proof}

\begin{remark}
The asymptotic behavior \eqref{small t} and \eqref{large t} is important for various applications of the heat kernel. Namely, the upper half-plane is the universal cover for the Riemann surface $\Gamma\mathbb{H}$, where $\Gamma$ is a co-finite Fuchsian group of the first kind.

For example, \eqref{small t} and \eqref{large t} plus an additional consideration related to the small $t$ asymptotics suffices to apply e.g. \cite[Theorem 5.1]{JLa03} to $K_{\mathbb{H}}\left( t,r\right)$.

In the case when the surface is compact and the multiplier system is well chosen (see e.g. \cite[discussion on pp. 335-337]{He83}), the spectral expansion of the heat kernel, together with \eqref{small t} and \eqref{large t} are sufficient to deduce the meromorphic continuation of the spectral zeta function, following the approach described in \cite{Sarnak}.

We leave those applications to a subsequent paper.
\end{remark}

\end{document}